\documentclass{amsart}
\usepackage{graphicx}
\usepackage{multicol,multirow}
\usepackage{amsmath,amssymb,amsfonts}
\usepackage{mathrsfs}
\usepackage{amsthm}
\usepackage{rotating}
\usepackage{appendix}
\usepackage[numbers]{natbib}
\usepackage{
amsthm,
amscd,
multicol,
wrapfig,
systeme, 
geometry,
mathtools,
subfiles,
import,
enumerate,
subcaption,
comment,
microtype,
orcidlink
}

\theoremstyle{theorem}
\newtheorem{theorem}{Theorem}[section]
\newtheorem{lemma}[theorem]{Lemma}
\newtheorem{corollary}[theorem]{Corollary}
\newtheorem{proposition}[theorem]{Proposition}
\theoremstyle{remark}

\newtheorem{remark}[theorem]{Remark}
\newtheorem{example}[theorem]{Example}

\numberwithin{equation}{section}

\DeclareMathOperator{\dist}{dist}
\DeclareMathOperator{\conv}{conv}
\DeclareMathOperator{\ee}{e}

\newcommand{\Comp}{\mathbb C}
\newcommand{\norm}[1]{\left\| #1\right\|}

\title[The algebraic numerical range as a spectral set]{The algebraic numerical range as a spectral set in Banach algebras}

\author{Hanna Blazhko}
\author{Daniil Homza}
\author{Felix L.~Schwenninger}
\author{\\Jens de Vries}
\author{Micha\l{} Wojtylak}

\address{Jagiellonian University, Faculty of Mathematics and Computer Science, Krak\'ow, Poland}
\email{hanna.blazhko@student.uj.edu.pl} 
\email{daniil.homza@student.uj.edu.pl}
\email{michal.wojtylak@uj.edu.pl}

\address{University of Twente, Department of Applied Mathematics, Enschede, The Netherlands}
\email{f.l.schwenninger@utwente.nl}
\email{j.devries-4@utwente.nl}



\begin{document}

\begin{abstract}
    {We investigate when the algebraic numerical range is a $C$-spectral set in a Banach algebra. While providing several counterexamples based on classical ideas as well as combinatorial Banach spaces, we discuss positive results for matrix algebras and provide an absolute constant in the case of complex $2\times2$-matrices with the induced $1$-norm. Furthermore, we discuss positive results for infinite-dimensional Banach algebras, including the Calkin algebra.}
\end{abstract}
\maketitle

\section{Introduction}

Since the discovery of the von Neumann inequality the theory of spectral sets has evolved in many directions, see e.g. \cite{badea2013spectralsets} for a broad overview. Our interest will lie in $C$-spectral sets. recall that a compact, simply connected set $\Omega$ is a \textit{$C$-spectral set} ($C>0$) for the operator $T$ if it contains the spectrum of $T$ and satisfies the inequality
\begin{equation}\label{Cspectral}
\norm{p(T)}\leq C\sup_{z\in\Omega} |p(z)|,\quad p\in\Comp[z],
\end{equation}
(see Section~\ref{sec:SC} for details). 
A particularly well-studied $C$-spectral set for a Hilbert space operator $T$ is the numerical range $
W(T) := \{ \langle Tx, x \rangle : \langle x, x \rangle = 1 \}.
$
The first result underlining the role of the numerical range in this context is the seminal work by Delyon--Delyon \cite{delyondelyon1999}, in which the constant $C$ depends on $W(T)$.  In 2007 Crouzeix \cite{Crouzeix2007} showed that there exists a universal constant $C$   {between} $2$ and $11.08$ such that the numerical range of $T$ is a $C$-spectral set. 
Crouzeix's conjecture \cite{Crouzeix2004}, stating that the optimal constant for $C$ is $2$, remains open to this day.
  {The best estimate so far} of the constant $C\leq   1 + \sqrt{2} $ was provided by Crouzeix and Palencia \cite{Crouzeix2017} in 2017, see also \cite{ransford2018remarks,schwenninger2024abstract} for shorter proofs. In due course the conjecture was shown to be true in certain special cases, see, e.g., \cite{bickel2023crouzeix,bickel2024blaschke,bickel2020crouzeix,chen2024gmres,choi2013proof,crouzeix2024numerical,o2023crouzeix}. {   We also mention recent results  \cite[Theorem 2, Proposition 26]{malman2024double} showing that the value $1+\sqrt{2}$ is not attained for any Hilbert space operator.}

Meanwhile, the theory of spectral sets in Banach spaces appears to be much more demanding. Recall that the von Neumann inequality can be restated as saying that the closed unit disk  
is a $1$-spectral set for any contraction on a Hilbert space.   {In contrast to this, there are Banach-space contractions for which the closed unit disk fails to be a $C$-spectral set for every $C > 0$.} 
This is usually argued in the literature by saying that there exists an operator with unit norm which is not \emph{polynomially bounded}, see e.g.\ \cite[Section 4]{Cohen2018} or \cite{lebow1968power,gillespie1975spectral,Gillespie2015}. 
On the other hand,  Katsnelson and Matsaev demonstrated in \cite{KM66} that for any contraction $T$ on any Banach space the disk $3\overline{\mathbb{D}}$ is a 1-spectral set, and furthermore, that the constant 3 is sharp. In fact, this result is a trivial consequence of a much older inequality attributed to Bohr, see \cite{Bohr14}, for disk algebra functions, see also \cite{paulsen2022} for a recent short proof {   and \cite{knese2024three} for multivariate analogues.}

This paper deals with the question to what extent the fact that the numerical range is $C$-spectral for Hilbert spaces operators can be generalized to Banach algebras. 
The most suitable generalization of the numerical range seems to be the \textit{algebraic numerical range} of an element $T$ of a Banach algebra $\mathcal A$, i.e., the set
\begin{align*}
V(T) := \{\phi(T): \phi\in \mathcal A', \|\phi\|=1=\phi(I) \},
\end{align*}
where  $I$ stands for the unit and $\mathcal{A}'$ stands for its dual space. Like in the Hilbert space setting, the algebraic numerical range contains the spectrum and is contained in the disk of radius $\norm T$.
It is also clear from the definition that $V(\alpha T)=\alpha T$ for $\alpha\in\Comp$. Hence, one may assume without loss of generality that $T$ is of norm one and consequently $V(T)$ is then contained in the closed unit disk. 
Note that  for operators with unit norm $C$-spectrality of $V(T)$ trivially implies polynomial boundedness. 
The above mentioned examples on Banach spaces thus provide operators for which \eqref{Cspectral} with $\Omega=V(T)$ does not hold for any $C>0$.  
We will strengthen this by showing that there even exists a polynomially bounded operator $T$ with unit norm such that $V(T)$ is not $C$-spectral for any $C$, see Example \ref{ex:polybddinfspectral}.

The paper is organized as follows.  In Section~\ref{s:prel} we review preliminary properties of the algebraic numerical range.
In Section~\ref{sec:SC} we introduce the spectral constant of the numerical range and discuss its general properties. 
In Section~\ref{s:MA} we discuss matrix algebras.  
In Section~\ref{s:fin} we focus on some special classes of Banach algebras for which the   {numerical-range} spectral constant is finite: Continuous function algebras and Calkin algebras. In particular, this shows that Crouzeix's conjecture on Hilbert space operators can be rephrased based on the \emph{essential numerical range} instead of the numerical range. In Section~\ref{sec:lp} we provide several examples in which the spectral constant of the numerical range is infinite. Namely, we consider the shift operators in $\ell^p$ and a  cut-shift in a combinatorial Banach space.  In Section~\ref{s:l1} we study in detail the case of $2\times 2$-matrices for $\ell_1$ induced norm and show that the spectral constant of any $2\times2$-matrix is bounded above by $13$.

The striking fact about Crouzeix's result is that the numerical range is a $C$-spectral set with an absolute constant $C$. In particular, this shows that the constant can be bounded independently of the dimension of the Hilbert space. While our results show that the latter is not true for general Banach algebras, we prove that for the specific case of $\mathcal{B}(\mathbb{C}^2)$ with the induced $1$-norm, there also exists a uniform bound on the spectral constant. It is interesting to note that our proof strongly exploits a nice representation, of interest in its own right, of the corresponding (algebraic) numerical range for the particular space. This, in a way, relates to Crouzeix's initial result on $2\times 2$-matrices \cite{Crouzeix2004} (with the optimal constant $2$), strongly resting on the fact that the (Hilbert space) numerical range is an ellipse in that case.

\section{Preliminaries}\label{s:prel}

Let $\mathcal{A}$ be a complex Banach algebra with a unit $I$, and let $ \mathcal{A}'$ denote the set of continuous linear functionals on $\mathcal A$ (dual space).   For $T \in \mathcal{A}$, the \textit{algebraic numerical range} $V(T)$ is defined as 
$$
V(T) := \{\phi(T): \phi\in \mathcal A', \|\phi\|=1=\phi(I) \}.
$$
The set $V(T)$ is compact and convex. We denote by $\mathcal B(X)$ the algebra of bounded linear operators on a Banach space $X$. Recall for $\mathcal A$ being the algebra $\mathcal B(H)$ of bounded operators on a Hilbert space $H$, the algebraic numerical range coincides with the closure of the usual numerical range, cf.~\cite[Th.6]{stampfli1968}.

 To specify the algebra with respect to which $V(T)$ is defined, we will sometimes write $V(T, \mathcal{A})$. Usually, we will do this in connection with the following important observation \cite[Th. 4]{bonsall71BookI}: If $\mathcal{B}\subseteq\mathcal{A}$ is a subalgebra 
of $\mathcal{A}$ sharing the same unit, then 
 \begin{equation}\label{VBVA}
V(T,\mathcal B)=V(T,\mathcal A)
 \end{equation}
 for all $T\in\mathcal{B}$. Further, by $\nu(T)$ we define the \textit{algebraic numerical radius} 
$$ 
\nu(T) := \sup \{ |z| : z \in V(T) \}.
$$
For the basic properties of the algebraic numerical range we refer the reader to \cite{bonsall71BookI}  {. Below} we only highlight the ones that are most relevant for our purposes.  For further results on equivalent definitions and geometry see, e.g., \cite{bonsall73BookII,lumer61,stampfli1968}.

  {The algebraic numerical range always contains the spectrum
\begin{align*}
    \sigma(T):=\{\lambda\in\mathbb{C}: \lambda I-T \ \text{is not invertible}\}.
\end{align*}
In the commutative case the inclusion goes easily by characters, for the general case we refer to \cite[Th. 6]{bonsall71BookI} and \cite[Th.1]{stampfli1968}.} %If $\mathcal{A}$ is commutative, then $\sigma(T)$ consists precisely of the elements $\phi(T)$ where $0\neq\phi\in\mathcal{A}'$ is multiplicative (and the conditions $\|\phi\|=1=\phi(I)$ are automatically satisfied).}
%To see this, recall that a $\phi\in\mathcal{A}'$ is called a \textit{character} if it is non-zero and multiplicative. Any character $\phi\in\mathcal{A}'$ automatically fulfills the conditions $\|\phi\|=1=\phi(I)$ and elementary Banach algebra theory tells us that $\lambda\in\sigma(T)$ if and only if there exists a character $\phi\in\mathcal{A}'$ such that $\lambda=\phi(T)$. It follows that indeed $\sigma(T)\subseteq V(T)$, also see \cite[Th. 6]{bonsall71BookI} and \cite[Th.1]{stampfli1968}. 
In particular, the spectral radius $\rho(T):=\sup_{\lambda\in\sigma(T)}|\lambda|$ satisfies $\rho(T)\leq\nu(T)$.
Further, for any $T\in\mathcal{A}$ it holds that
\begin{equation}\label{nuenu}
\nu(T)\leq \| T \| \leq \ee \nu(T),
\end{equation}
see \cite[S.4 Th. 1, p.34]{bonsall71BookI}, and if $\mathcal A$ is a $C^*$-algebra the constant $\ee$ can be improved to $2$, see \cite{NagyFoiasbook}.

The following equality shown in \cite[Th.4]{stampfli1968}, see also \cite{hildebrandt66}, will be of particular importance when searching for an explicit form of the algebraic numerical range:

\begin{equation}\label{circleset}
    V(T) = \bigcap_{\lambda \in \mathbb{C}} D(-\lambda,  \|T+\lambda I \|) ,
    \end{equation}
where $D(z,r)\subset \mathbb C$ denotes a closed disk of radius $r$ centered at $z$. An immediate consequence is that 
$V(T)$ is a compact convex subset of $\Comp$.
An even more useful result for plotting approximations of $V(T)$ is \cite[Th. 2.5]{bonsall71BookI}, which states that the algebraic numerical range $V(T)$ can be represented as an intersection of hyperplanes:
\begin{equation}\label{hyperset}
    V(T)=\bigcap_{\theta \in [0,2\pi)} H_{\theta} (T),
\end{equation}
where
\begin{equation}
    H_{\theta} (T):= \{ e^{i\theta}\alpha: \ \alpha \in \mathbb{C}, \ \text{Re} (\alpha) \leq r_{\theta} (T) \}
\end{equation}
with
\begin{equation}\label{r-theta}
    r_{\theta} (T) := \inf_{t \in [0,\infty)} \left\{  \|e^{-i\theta}T +tI \| - t  \right\}.
\end{equation}
Recall that the algebraic numerical range is preserved under affine transformations:
\begin{equation}\label{alphabeta}
V(\alpha T+\beta I)=\alpha V(T)+\beta
\end{equation}
for all $\alpha,\beta\in\mathbb{C}$, $T\in\mathcal A$ (however, in general not under polynomial transformations). This, together with \eqref{nuenu}, implies that
\begin{equation}\label{singleton}
V(T)=\{\lambda_0\} \ \text{if and only if} \ T=\lambda_0 I.
\end{equation}

Another simple yet useful observation we provide along with the proof. We write $d_{\mathrm{H}}$ for the Hausdorff metric on the set of non-empty compact subsets of the complex plane.

\begin{lemma}\label{Hau}
Let $\mathcal{A}$ be a unital Banach algebra. For all $S,T\in\mathcal{A}$ it holds that
\begin{align*}
	d_{\mathrm{H}}(V(S),V(T))\leq\|S-T\|.
\end{align*}
In particular, the mapping $T\mapsto V(T)$ is uniformly continuous.
\end{lemma}
\begin{proof}
For each $\phi\in\mathcal{A}^{*}$ with $\phi(1)=1$ and $\|\phi\|=1$ we have
\begin{align*}
	\operatorname{dist}(\phi(S),V(T))\leq|\phi(S)-\phi(T)|\leq\|S-T\|
\end{align*}
and, likewise,
\begin{align*}
	\operatorname{dist}(V(S),\phi(T))\leq|\phi(S)-\phi(T)|\leq\|S-T\|.
\end{align*}
So by definition of the Hausdorff distance we have
\begin{align*}
d_{H}(V(S),V(T))=\max\Big\{\sup_{z\in V(S)}\operatorname{dist}(z,V(T)),\sup_{z\in V(T)}\operatorname{dist}(V(S),z)\Big\}\leq\|S-T\|
\end{align*}
as desired.
\end{proof}
We conclude this list of properties of $V(T)$ with the following estimate for the growth of the resolvent, cf.\ \cite[Lemma 1]{stampfli1968},
\begin{equation}\label{resineq}
    \norm{(\lambda I -T)^{-1}}\leq \dist(\lambda,V(T))^{-1},\quad \lambda\notin V(T).
\end{equation}

\section{ Spectral constants}\label{sec:SC}
Let us turn to the main object of the study. Hereinafter let $\mathcal A$ denote a complex Banach algebra with the unit $I$. 
Let $\Omega$ be an open or closed, simply connected bounded set. We say that $\Omega$ is a $C$-spectral set for $T\in\mathcal A$ if $\Omega$ contains the spectrum of $T$ and \eqref{Cspectral} holds. Note that it is also possible to define $C$-spectrality for multiply connected and unbounded sets, replacing polynomials by rational functions. However, this is not relevant for the purposes of the current paper. 

We say that an element $T\in\mathcal A$ is \emph{polynomially bounded} if there exists $C>0$ such that
$$
\norm{p(T)}\leq C \sup_{|z|\leq1} |p(z)|,
$$
for all $p\in \mathbb{C}[z]$. 
We will typically investigate whether $\norm T\overline{\mathbb D}$ is a $C$-spectral set for $T$, which is equivalent to $T/\norm T$ being polynomially bounded.

For an element $T\in\mathcal A$ we consider the linear mapping
$$
\Phi_T: \Comp[z]\ni f \mapsto f(T)\in \mathcal A.
$$
and we define the \emph{numerical  {-}range spectral constant of $T$} as
\begin{equation}\label{PsiT}
 \Psi(T):=\Psi(T,\mathcal A):=\inf\{C>0: \| p(T) \| \leq C \sup_{V(T)} | p|,\ p\in\Comp[\lambda]  \},
\end{equation}
with $\inf\emptyset=\infty$. If finite, $\Psi(T)$ it it the smallest $C\geq0$ for which \eqref{Cspectral} holds with $\Omega=V(T)$.
Note that for  $T=\lambda_0 I$  ($\lambda_0\in\Comp$) we have   $\Psi(T)=1$. For $T$ not being the multiple of the unit $I$ we have that $V(T)$ is a convex set that is not a singleton, cf.\ \eqref{singleton}.  In this case $\Psi(T)$ is the operator norm of $\Phi_T$  with respect to the supremum norm on $V(T)$ and the usual norm on $\mathcal A$, provided $\Phi_T$ is bounded.   {As was already pointed out} in the introduction, if $T$ is a contraction that is not polynomially bounded, we have $\Psi(T,\mathcal{B}(X))=\infty$, due to the first inequality in \eqref{nuenu}.

Finally, 
we define the \emph{(algebraic) numerical  {-}range spectral constant} of the algebra $\mathcal A$ by
\begin{equation}\label{PsiA}
\Psi_{\mathcal A}:=\sup_{T\in\mathcal A} \Psi(T) \in[0,\infty].
\end{equation}
Later on we will skip the adjective `algebraic' for brevity. Directly from \eqref{VBVA} we get that if $\mathcal B$ is a subalgebra of $\mathcal{A}$ sharing the same unit  {. Then} 
\begin{equation}\label{PsiAB}
    \Psi(T;\mathcal B)=\Psi(T;\mathcal A),\quad \Psi_{\mathcal B}\leq\Psi_{\mathcal A},
\end{equation}
note that we do not assume that any of these numbers are finite. Let us discuss the elementary properties of the function $T\mapsto\Psi(T)$.

\begin{proposition}\label{semicont}
The function $\mathcal A \ni T\mapsto\Psi(T)\in[0,+\infty]$ has the following properties:
\begin{enumerate}[\rm (i)]
\item $\Psi(\cdot)$ is bounded from below by 1;
\item $\Psi(\cdot)$ is lower semi-continuous;
\item $\Psi(\cdot)$ attains its minimum on every compact set;
\item $\Psi(\alpha T+\beta I)=\Psi(T)$ for every $\alpha\in\Comp\setminus\{0\}$, $\beta\in\mathbb{C}$, $T\in\mathcal A$;
\item for any $C\in[1,\infty)$,  the set $\{ T\in\mathcal A:  \Psi(T)>C \}$ is open and if it is nonempty then $\lambda_0 I$ belongs to its boundary for any $\lambda_0\in\Comp$;
\item $\Psi(\cdot)$ is not continuous, unless it is constantly equal to one.
\end{enumerate}
\end{proposition}

\begin{proof}Taking the polynomial $p(z)=z$ shows (i).

(ii) Assume the contrary, i.e.\ that $\Psi(\cdot)$ is not lower semicontinuous at some point $T$ in $\mathcal{A}$.  Then there exists  $C\in(0,\infty)$, an $\varepsilon>0$ and a convergent sequence $T_{k}\to T$ such that $\Psi(T)>C+\varepsilon$ and $\Psi(T_{k})\leq C$ for all $k\geq1$. Note that $C$ is finite, regardless of whether $\Psi(T)$ is finite or not.
Further, there  exists a polynomial $p\in\mathbb{C}[z]$ such that $\sup_{V(T)}|p|=1$ and $\|p(T)\|\geq C+\varepsilon$. Now fix a number 
\begin{align*}
0<\delta<\sqrt{\frac{C+\varepsilon}{C}}-1.
\end{align*}
Since $V(T_{k})\to V(T)$ with respect to the Hausdorff metric, there exists an integer $N_{1}\geq1$ such that $\sup_{V(T_{k})}|p|\leq1+\delta$ for all $k\geq N_{1}$. Since $p(T_{k})\to p(T)$, there also is an integer $N_{2}\geq1$ such that $\|p(T_{k})\|\geq(1+\delta)^{-1}\|p(T)\|$ for all $k\geq N_{2}$. Thus, for $k\geq N:=\max\{N_{1},N_{2}\}$ we have
\begin{align*}
\Psi(T_{k})(1+\delta)\geq\Psi(T_{k})\sup_{V(T_{k})}|p|\geq\|p(T_{k})\|\geq\frac{\|p(T)\|}{1+\delta}\geq\frac{C+\varepsilon}{1+\delta}.
\end{align*}
However, this implies that $\Psi(T_{k})>C$ for all $k\geq N$, which is a contradiction.

(iii) follows directly from (ii). 
To see (iv) define the mapping $\Lambda :\Comp[z]\ni p (z)\mapsto p(\alpha z+\beta)\in\Comp[z]$  for fixed $\alpha\neq 0$ and $\beta\in\Comp$ and note that it is bijective. 
Hence, by \eqref{alphabeta},
$$
\Psi(\alpha T+\beta I)= \inf\{C>0: \| \Lambda( p) (T) \| \leq C \sup_{V(T)} |\Lambda(p)|,\ p\in\Comp[\lambda]\setminus\{ 0\}  \}=\Psi(T).
$$

(v)   {Let us fix $C>1$. Then} the set $\mathcal S:=\{ T\in\mathcal A:  \Psi(T)> C \}$ is open as the function $\Psi(\cdot)$ is lower semi-continuous. Fix $\lambda_0\in\Comp$ and take any $T$ with  $\Psi(T)=C_0\in [C,\infty]$. Then $\Psi(\alpha T+ \lambda_0 I)=C_0$ for any $\alpha\neq 0$, hence $\lambda_0 I$ is in the closure of $\mathcal S$. As $\Psi(I)=1$ we have that $\lambda_0I$ is on the boundary of $\mathcal S$. 
Statement (vi) is also a consequence of this reasoning. 
\end{proof}

We provide one more elementary property of the function $\Psi(\cdot)$. Given a Banach space $X$ and its  dual  $X'$ we naturally define the adjoint $T'\in\mathcal B(X')$ of $T\in\mathcal B(X)$. Recall that the map $T\mapsto T'$ is a linear isometry. The following lemma shows that the algebraic numerical range and numerical  {-}range spectral constants  \eqref{PsiT} coincide for $T$ and $T'$.
	\begin{lemma}\label{dualcro}
		Let $X$ be any Banach space. For any $T\in\mathcal{B}(X)$ the following is true:
		\begin{enumerate}
			\item[\rm (i)]	$V(T',\mathcal{B}(X'))=V(T,\mathcal{B}(X))$;
			\item[\rm (ii)] $\Psi(T',\mathcal{B}(X'))=\Psi(T,\mathcal{B}(X))$;
         \item[\rm (iii)] $\Psi_{\mathcal{B}(X)}\leq \Psi_{\mathcal{B}(X')}$, with the inequality being an  equality for reflexive $X$.
		\end{enumerate}
	\end{lemma}
	\begin{proof}
			(i) Since for any $\theta\in\mathbb{R}$, $\alpha>0$ we have
		\begin{align*}
			\|I_{X'}+\alpha e^{i\theta}T'\|_{\mathcal{B}(X')}=\|I_{X}+\alpha e^{i\theta}T\|_{\mathcal{B}(X)},
		\end{align*}
		 the result follows from \eqref{hyperset}.
		
		(ii) For $\zeta\in\mathbb{C}$ we trivially have
    \begin{align*}
    \Psi(\zeta I_{X'},\mathcal{B}(X'))=\Psi(\zeta I_{X},\mathcal{B}(X))=1.
    \end{align*} 
    Assume that $T$ is not a scalar multiple of $I_{X}$. By (i) we have 
		\begin{align*}
			\Psi(T',\mathcal{B}(X'))&=\sup_{\substack{f\in\mathbb{C}[z] \\ f\neq0}}\frac{\|f(T')\|_{\mathcal{B}(X')}}{
				\|f\|_{\infty,V(T',\mathcal{B}(X'))}}\\
			&=\sup_{\substack{f\in\mathbb{C}[z] \\ f\neq0}}\frac{\|f(T)'\|_{\mathcal{B}(X')}}{\|f\|_{\infty,V(T',\mathcal{B}(X'))}}\\
			&=\sup_{\substack{f\in\mathbb{C}[z] \\ f\neq0}}\frac{\|f(T)\|_{\mathcal{B}(X)}}{\|f\|_{\infty,V(T,\mathcal{B}(X))}}\\
			&=\Psi(T,\mathcal{B}(X))
		\end{align*}
		as desired. 
  
The inequality in statement (iii) is now obvious. Further, if $X$ is reflexive, then $T\mapsto T'$ is surjective and the equality follows.
	\end{proof}

  {The $\varepsilon$-hull ($\varepsilon>0$) of a compact set $S \subseteq \Comp$ is  defined as
$S_\varepsilon = \{x \in \Comp \mid  \dist(x,S) \leq \varepsilon \} $.}
As it was already mentioned in the introduction, the constant $\Psi_{\mathcal A}$ might be infinite for some algebras.
However, both the $\varepsilon$-hull of the algebraic numerical range $V(T)_{\varepsilon}$ and the disk $(1+\varepsilon)\norm T \overline{\mathbb D}$ are $C$-spectral sets, as the following proposition shows. 
\begin{proposition}\label{prop:epsilonball}
For  every Banach space $X$, any operator $T\in\mathcal{B}(X)$ and $\varepsilon>0$ we have 
\[
\|p(T)\|\leq \left(1+ \frac 1{2\varepsilon}\right) 
\sup_{z\in V(T)_{\varepsilon d}}|p(z)|,\quad p\in\mathbb{C}[z],
\]
where $d$ is the diameter of $V(T)$  {. In particular, if $\norm T\leq 1$, then} $V(T)_1$ is a $2$-spectral set. 
Furthermore, for arbitrary $\norm T$,
\[\|p(T)\|\leq \frac{1+\varepsilon}{\sqrt{\varepsilon(2+\varepsilon)}} 
\sup_{|z|\leq (1+\varepsilon)\|T\|}|p(z)|,\quad p\in\mathbb{C}[z].\]
\end{proposition}
\begin{proof}
Since $V(T)_{\varepsilon d}$ is convex, its compact boundary $\partial V(T)_{\varepsilon d}$ is locally the graph of a Lipschitz function and hence it is rectifiable. The first statement follows now directly from the Cauchy integral formula (or, more precisely, from the Riesz-Dunford calculus) and from the estimate \eqref{resineq}.
The fact that $V(T)_1$ is a $2$-spectral set follows from taking $\varepsilon=1/2$, so that $\varepsilon d\leq \varepsilon\cdot 2\norm T\leq 1$.

To see the second statement let $T\in\mathcal{B}(X)$, $\varepsilon>0$ and $p(z)=\sum_{k=0}^{n}a_{k}z^{k}$. Further, let $c=(1+\varepsilon)\|T\|$ and consider 
\[\
\left\|\sum_{k=0}^{n}a_{k}T^{k}\right\|=\left\|\sum_{k=0}^{n}a_{k}c^{k}(\tfrac{1}{c}T)^{k}\right\|\leq \left(\sum_{k=0}^{n}|a_{k}c^{k}|^{2}\right)^{\frac{1}{2}}\left(\sum_{k=0}^{\infty}(1+\varepsilon)^{-2k}\right)^\frac{1}{2}=\|f\|_{H^{2}(\mathbb{D})}C_{\varepsilon},
\]
where $f(z)=\sum_{k=0}^{n}a_{k}c^{k}z^{k}$ for $z\in\mathbb{D}$, the norm $\norm{\cdot}_{H^{2}(\mathbb{D})}$ refers to the norm of the Hardy space $H^{2}(\mathbb{D})$ and $C_{\varepsilon}$ to the constant above. Finally, we use that  $H^{\infty}(\mathbb{D})$, the space of bounded analytic functions on $\mathbb{D}$, embeds continuously into $H^{2}(\mathbb{D})$. More precisely, 
$$
\|f\|_{H^{2}(\mathbb D)}\leq \sup_{|z|<1}|f(z)|=\sup_{|z|<c}|p(z)|,
$$
which finishes the proof. 
\end{proof}
\begin{remark}
    Clearly the first estimate in Proposition \ref{prop:epsilonball} is not sharp, as for 
    $T$ with $V(T)$ being the disk of radius $\norm T$ (e.g.\ when $T$ is the forward shift on $\ell^1$) the second bound is better.
    Furthermore, for large $\varepsilon$ the second bound is also not optimal. For example, taking $\varepsilon=2$ in Proposition \ref{prop:epsilonball} leads to the estimate $\|f(T)\|\leq C\sup_{|z|\leq 3\|T\|}|f(z)|$ with $C=3/\sqrt{8}\geq1.06$ for all polynomials $f$, while it is known that this inequality  holds even with $C=1$, see \cite{KM66,Bohr14,paulsen2022}.
\end{remark}

\section{    Finite dimensional  algebras }\label{s:MA}

  {
Let us recall that an element $A$ of a Banach algebra is called \textit{algebraic} if there exists a polynomial $p$ such that $p(A)=0$. This is equivalent to the fact that $A$ generates a finite-dimensional subalgebra. The central theorem of this section is the following. 
\begin{theorem}\label{findimth}
    For any unital Banach algebra $\mathcal{A}$ the numerical-range spectral constant $\Psi(T, \mathcal{A})$ is finite for all algebraic elements $T$ of $\mathcal{A}$.
\end{theorem}}

For the proof we need the following lemma, 
based on the estimation of the resolvent growth near the algebraic numerical range, see \cite{stampfli1968}.

\begin{lemma}\label{eigbd}
Let $\norm{\cdot}$ be any unital submultiplicative norm on $\mathbb{C}^{n\times n}$. Let $T\in\mathbb{C}^{n\times n}$ and $\lambda\in\sigma(T)$ be given. If $\lambda$ belongs to the boundary $\partial V(T)$, then $\lambda$ is a semisimple eigenvalue.
\end{lemma}
\begin{proof}
Suppose that $\lambda\in\partial V(T)$ is not semisimple. 
  {Then the Jordan normal form of $T$ contains} a Jordan block $J$ corresponding to $\lambda$ of size $s$, where $1 < s \leq n$. Since $\lambda\in\partial V(T)$, we can find a sequence $(\lambda_{i})_{i\in\mathbb{N}}$ in $\mathbb{C}\setminus V(T)$ such that $\lambda_{i}\to\lambda$ and $d_{i}:=d(\lambda_{i},V(T))=|\lambda_{i}-\lambda|$. Let $e_{s}$ be the $s$-th standard basis vector of $\mathbb{C}^{s}$ and   
{let $\norm{\cdot}_1$ denote the operator norm induced by the $\ell_1$-norm on $\Comp^s$}. Note that

\begin{align*}
\|(\lambda_{i}-J)^{-1}e_{s}\|_{1}=
\sum_{k=1}^{s}\frac{1}{d_{i}^{k}}.
\end{align*}
On the other hand, 
\begin{align*}
\|(\lambda_{i}-J)^{-1}e_{s}\|_{1}\leq\|(\lambda_{i}-J)^{-1}\|_{1}\leq C\|(\lambda_{i}-T)^{-1}\|\leq\frac{C}{d_{i}},
\end{align*}
where $C>0$ depends only on the similarity transformation for the Jordan decomposition of $T$ and the equivalence between the norms $\norm{\cdot}_{1}$ and $\norm{\cdot}$. The last inequality follows directly from \eqref{resineq}. We deduce that
\begin{align*}
1+\frac{1}{d_{i}}+\ldots+\frac{1}{d_{i}^{s-1}}\leq C,
\end{align*}
which contradicts $d_{i}\to0$.
\end{proof}

\begin{proof}

  {(of Theorem~\ref{findimth}) Since $T$ generates a finite-dimensional subalgebra, due to \eqref{PsiAB} it is enough to consider only finite-dimensional unital Banach algebras $\mathcal{A}$. Furthermore, any such algebra $\mathcal{A}$ can be isometrically embedded in $\mathcal{B(A)}$. Therefore, once again by \eqref{PsiAB}, it is enough to consider only unital Banach algebras of operators on finite-dimensional spaces. In other words, it is enough to show that for fixed $n\in\mathbb N$, fixed $T\in\Comp^{n,n}$ and fixed unital Banach algebra norm $\norm{\cdot}$ on $\Comp^{n,n}$ one has $\Psi(T, \mathcal B(\Comp^n, \norm\cdot)) < \infty$. 
Note that for some constant $C$ (depending on the norm and hence implicitly on the dimension $n$)  we have that
 $\norm{p(T)}\leq C \norm{p(T)}_2$ for any polynomial $p(z)$. Hence, we subsequently reduce the proof to  showing that}
 $$
\norm{p(T)}_2\leq C_2 \sup_{V(T, \mathcal B(\Comp^n, \norm\cdot))}|p|,\quad p\in\Comp[z],
 $$
 with some constant $C_2$ possibly dependent on $T$. 
 Let $T_J$ denote the Jordan form of $T$ and let $S$ denote the corresponding similarity transformation. We write
 $$
T_J=\text{diag}(\lambda_1,\dots,\lambda_r)\oplus R,\quad \sigma(T)=\{\lambda_1,\dots,\lambda_r\}\cup\sigma(R),
 $$
 where $\lambda_1,\dots,\lambda_r$ are the semisimple eigenvalues written with their multiplicities and $R$ consists of all nontrivial Jordan blocks  (possibly one of   {these} two parts constituting $T_J$ might be void). 
Recall that the algebraic numerical range $V(T, \mathcal B(\Comp^n, \norm\cdot))$ has the property that all eigenvalues of $T$ on its boundary are semisimple by Lemma ~\ref{eigbd}.
Hence, the eigenvalues of $R$ (if there are any) lie inside the interior of $V(T, \mathcal B(\Comp^n, \norm\cdot))$, by  Lemma~\eqref{eigbd}. The boundary $\partial V(T,\mathcal{B}(\Comp^n,\norm{\cdot}))$ is rectifiable. Estimating in a routine way the Cauchy integral formula we receive $\norm{p(R)}\leq C_1 \sup_{V(T, \mathcal B(\Comp^n, \norm\cdot))}|p|$ for any polynomial $p$ and some constant $C_1$, depending on the maximum of the norm of the resolvent of $R$ on $\partial V(T, \mathcal B(\Comp^n, \norm\cdot))$. 
 Therefore,
 $$
\norm{p(T)}_2\leq \norm{S}_2\norm{S^{-1}}_2 \norm{p(T_J)}_2
 $$
 and 
\begin{align*}
    \norm{p(T_J)}_2&=\max ( |p(\lambda_1)|,\dots, |p(\lambda_r)|,\norm{ p(R)}) \\
    &\leq \max(1,C_1)  \sup_{V(T, \mathcal B(\Comp^n, \norm\cdot))}|p|,
\end{align*}
for all $p\in\Comp[z]$, from which we obtain the constant $C_2$. 

\end{proof}

Immediately we provide an example that the numerical  {-}range spectral constant $\Psi_{\mathcal B(\Comp^n,\norm\cdot_p)}$ of the matrix algebra with the $\ell^p$-induced norm depends on the dimension $n$. It should not come as surprise that the `bad' matrix will be the Jordan block
\begin{align*}
	J_{n}:=\begin{bmatrix}
		0&1&&\\
		&0&\ddots&&\\
		&&\ddots&\ddots&\\
		&&&0&1\\
		&&&&0
	\end{bmatrix}\in\mathbb{C}^{n\times n}.
\end{align*}
The following facts we use to derive suitable operator 
norm estimates are    {well known} and   {can be traced} back to the works of Shapiro \cite{shapiro1952extremal} and Rudin \cite{rudin1959some}.
There exists $\Delta>0$ such that for every $n\in \mathbb{N}$ there exists a polynomial $f_{n}=\sum_{k=0}^{n-1}\alpha_{k}z^{k}$ of degree $n-1$, with coefficients $\alpha_k$ all equal to either $1$ or $-1$, and satisfying 
\begin{align}\label{Littlewood}
	|f_{n}(z)|\leq\Delta\sqrt{n},\quad |z|=1,\ n=1,2,\dots.
\end{align}
The best known constant is $\Delta=\sqrt 6$, see \cite{balister2019bounds}. We also refer to 
\cite{Balister2020} for a recent  solution of the related \emph{Littlewood conjecture}  \cite{Littlewood1966}, stating that there even exist such polynomials satisfying the lower bound
$	\delta\sqrt{n}\leq|f_{n}(z)|\leq\Delta\sqrt{n}$ with some constants $0<\delta\leq \Delta$ independent of $n$.

\begin{theorem}\label{Jordan}
Let $p\in[1,\infty]$ with H\"older conjugate $q\in[1,\infty]$. The numerical  {-}range spectral constant \eqref{PsiT} of the Jordan block in the $\ell^p$-induced norm satisfies
\begin{align*}
	\Psi(J_n,\mathcal B(\Comp^n,\norm\cdot_p))\geq  \sup_{\substack{f\in\mathbb{C}[z] \\ f\neq0}}\frac{\|f(J_{n})\|_{p}}{\|f\|_{\infty,\overline{\mathbb{D}}}}\geq \frac{1}{\sqrt{6}}n^{\max\{\frac{1}{p},\frac{1}{q}\}-\frac{1}{2}}
\end{align*}
for all $n\in\mathbb{N}$.
\end{theorem}
\begin{proof}
The first inequality follows from the fact that $\norm{J_n}\leq 1$. To see the second one observe that for any polynomial $f$ of degree $n-1$ with coefficients $\alpha_{0},\alpha_{1},\alpha_{2},\ldots,\alpha_{n-1}$ the polynomial functional calculus of $J_{n}$ is given by
\begin{align*}
    f(J_{n})=
	\begin{bmatrix}
		\alpha_{0}&\alpha_{1}&\alpha_{2}&\cdots&\cdots&\alpha_{n-1}\\
		&\alpha_{0}&\alpha_{1}&&&\vdots\\
		&&\ddots&\ddots&&\vdots\\
		&&&\ddots&\alpha_{1}&\alpha_{2}\\
		&&&&\alpha_{0}&\alpha_{1}\\
		&&&&&\alpha_{0}
	\end{bmatrix}.
\end{align*}
From this it quickly follows that
\begin{align*}
	\|f(J_{n})\|_{p}\geq\|f(J_{n})e_{n}\|_{p}=\begin{cases}
	(\sum_{j=0}^{n-1}|\alpha_{j}|^{p})^{\frac{1}{p}}&p<\infty\\
    \max_{j=0}^{n-1}|\alpha_{j}|&p=\infty
	\end{cases}
	\end{align*}
	and, under the Banach space isomorphism $(\mathbb{C}^{n},\norm{\cdot}_{p})'=(\mathbb{C}^{n},\norm{\cdot}_{q})$,
	\begin{align*}
    \|f(J_{n})\|_{p}=\|f(J_{n})'\|_{q}\geq\|f(J_{n})'e_{1}\|_{q}=\begin{cases}
	(\sum_{j=0}^{n-1}|\alpha_{j}|^{q})^{\frac{1}{q}}&q<\infty\\
    \max_{j=0}^{n-1}|\alpha_{j}|&q=\infty
	\end{cases}.
	\end{align*}
Applying this to the polynomials from  \eqref{Littlewood} 
we conclude that
\begin{align*}
\sup_{\substack{f\in\mathbb{C}[z] \\ f\neq0}}\frac{\|f(J_{n})\|_{p}}{\|f\|_{\infty,\overline{\mathbb{D}}}}\geq\frac{\|f_{n}(J_{n})\|_{p}}{\|f_{n}\|_{\infty,\overline{\mathbb{D}}}}\geq\frac{\max\{n^{\frac{1}{p}},n^{\frac{1}{q}}\}}{\sqrt{6}\sqrt{n}}=\frac{1}{\sqrt{6}}n^{\max\{\frac{1}{p},\frac{1}{q}\}-\frac{1}{2}}
\end{align*}
as desired.
\end{proof}

It remains unknown whether $\Psi_{\mathcal{A}}$ is finite for   {every} matrix algebra $\mathcal A$.

\section{
\texorpdfstring{Infinite-dimensional algebras with finite constant $\Psi_{\mathcal A}$}{Infinite-dimensional algebras with finite constant Ψ(A)}}\label{s:fin}

The numerical  {-}range spectral constant of the algebra $\mathcal B({H})$, where $H$ is a Hilbert space of any infinite dimension, has a special role. Namely, it follows from \cite{Crouzeix2007} that  
\begin{equation}\label{Psicrodef}
\Psi_{\mathcal B( H )} = \Psi_{\rm Cro}  :=\sup_{n\geq 1}  \Psi_{\mathcal B(\Comp^n,\norm{\cdot}_2)}<\infty.
\end{equation}
The universal constant $\Psi_{\rm Cro}$ appearing above is called the Crouzeix constant, see the introduction for a brief review on related results  {. Here} we recall that 
$2\leq\Psi_{\rm Cro}\leq 1+\sqrt 2$. 

It follows from \eqref{PsiAB}  that if $\mathcal A$ is a $C^*$-algebra, then, as a subalgebra of $\mathcal B(H)$ for some Hilbert space $H$, we obtain
\begin{equation}\label{CAC}
\Psi_\mathcal A\leq\Psi_{\rm Cro}.
\end{equation}
Further, there exist several sufficient conditions for embeddability of a given algebra (not necessarily a $*$-algebra) in $\mathcal B(H))$, see e.g. \cite{blecher1995completely,blecher2004operator}, guaranteeing in turn \eqref{CAC}.
Let us now present two instances where the numerical range  {-}spectral constant \eqref{PsiA} can be computed explicitly.

\begin{theorem}\label{th:functions}
Let $X$ be a compact space and let $\mathcal B$ be a Banach algebra. Let $C(X,\mathcal B)$ be the Banach algebra of $\mathcal B$-valued continuous functions on $X$ with the norm  
$\norm f:=\sup_{x\in X}\norm{ f(x)}_{\mathcal{B}}$. Then the corresponding numerical  {-}range spectral constants, defined in  \eqref{PsiA}, satisfy
$$
\Psi_{C(X,\mathcal B)}=\Psi_{\mathcal B},
$$
regardless of whether $\Psi_{\mathcal B}$ is finite or not.\newline
In particular,   for  any unital commutative $C^*$-algebra $\mathcal A$ the constant $\Psi_ \mathcal A$ equals 1.
\end{theorem}

\begin{proof}
Assume first that $\Psi_{\mathcal B}$ is finite. Let $f\in C(X,\mathcal B)$.   Note that each pair $(x,\phi)$, where $x\in X$ and $\phi\in\mathcal B'$ with $\norm{\phi}=1=\phi(I)$, constitutes a functional $f\mapsto \phi(f(x))$ in the dual space $C(X,\mathcal B)'$. Therefore,
 \begin{eqnarray*}
\norm{p(f)}_{C(X,\mathcal{ B})} &=& \sup_{x\in X} \|p(f) (x)\| \\
&=&  \sup_{x\in X} \|p(f (x))\| \\
&\leq& \Psi_{\mathcal B}\sup_{x\in X} \sup_{\substack{\phi\in \mathcal{ B}'\\ \norm\phi=1=\phi(I)}} |p(\phi(f(x))|
\end{eqnarray*}   
which shows  the inequality $\Psi_{C(X,\mathcal B)}\leq\Psi_{\mathcal B}$. The reverse inequality, and the case $\Psi_{\mathcal B}=\infty$ follow by identifying $\mathcal B$ with constant functions in $C(X,\mathcal B)$ and applying \eqref{PsiAB}.

   The second statement follows from identifying $\mathcal A$ with the algebra $ C(X,\Comp)$ and Proposition \ref{semicont}(i).
\end{proof}

Given a Hilbert space $ H $, the Calkin algebra is defined as the quotient 
$$
\mathcal{C}( H ):=\mathcal{B}( H )/\mathcal{K}( H ),
$$
where $\mathcal K( H )$ denotes the ideal of compact operators in the bounded linear operators $\mathcal{B}(H)$ on $H$. For an operator $T\in\mathcal B(H)$ its equivalence class will be denoted by $[T]$. In the proof below a mapping between two $C^*$-algebras is called an embedding if it is a linear isometry   {that} is multiplicative and preserves the involution and identity.

\begin{theorem}
For a separable   {infinite-dimensional} Hilbert space $H$ we have the following equality between the numerical  {-}range spectral constants \eqref{PsiA}
\begin{equation}\label{PsiPsi}
\Psi_{ \mathcal{C}( H ) }=\Psi_{\rm Cro}
\end{equation}  
Furthermore, for any $T\in\mathcal{B}(H)$ we have 
\begin{equation}\label{CleC}
\Psi( [T],\mathcal C(H) )\leq \Psi ( T+K,\mathcal B(H))
\end{equation}
for some $K\in\mathcal {K}(H)$. 
\end{theorem}

\begin{proof}
  {Let us recall that the quotient of a C*-algebra by a closed *-ideal is again a C*-algebra. Thus the Calkin algebra is a C*-algebra.} Then the inequality $ \Psi_{ \mathcal{C}( H ) }\leq\Psi_{\rm Cro}$ follows from \eqref{CAC} and \eqref{Psicrodef} via the fact that $\mathcal{C}(H)$ can be embedded in  a C*-subalgebra of $\mathcal B(\tilde H)$ for some Hilbert space $\tilde H$. The reverse inequality follows again  from \eqref{CAC} and \eqref{Psicrodef}  and the fact that the algebra $\mathcal B( \Comp^d, \norm{\cdot}_2)$  can be embedded in a $C^*$-subalgebra of the Calkin algebra $\Psi_{ \mathcal{C}( H ) }$. Below we present a simple proof of the latter fact, referring also to \cite{Farah2021} for a rich theory of embedings of $C^*$-algebras into the Calkin algebra. 

 Note that it is enough to embed $\mathcal B( \Comp^d, \norm{\cdot}_2)$ in $\mathcal C(H_0)$ for some separable Hilbert space $H_0$. 
 We define $H_{0}:=\ell^{2}\otimes\mathbb{C}^{d}$ and let $\pi (T)=[I_{\ell^{2}}\otimes T]\in\mathcal{C}(H_0)$. 
 Let $P_k$ be an orthogonal projection on the first $k$ basis vectors of $\ell^2$ and let $Q_k=(I_{\ell^2} - P_k)\otimes I_{\Comp^d}$. 
By Proposition 6 of \cite{Muller2010} we have that 
$$
\norm{ \pi(T) }_{\mathcal{C}(H_0)}=   \lim_{k\to\infty} \norm{ Q_k(I_{\ell^2}\otimes T) Q_k }_{\mathcal{B}(H_0)}=\norm{T}_{\mathcal B(H_0)}.
$$
Hence, the mapping $\pi$ is an isometry, it is also clearly linear, multiplicative, and preserves the adjoint and identity.

Let us show now \eqref{CleC}.
By \cite{Chui1977} (see also \cite{Muller2010}) there exists
	$K\in\mathcal{K}(H)$ such that $V(T+\mathcal{K}(H))=V(T+K)$. It follows that
	\begin{align*}
		\|p([ T] )\|_{\mathcal C(H)}&=\|[ p(T)]\|_{\mathcal C(H)}\\
		&\leq\|p(T+K)\|_{\mathcal{B}(H)}\\
		&\leq\Psi(T+K)\cdot {\textstyle \sup_{V(T+K)}|p|}\\
		&=\Psi(T+K)\cdot {\textstyle\sup_{V([T] )}|p|}
	\end{align*}
	for any polynomial $p$.
\end{proof}

\section{
\texorpdfstring{Algebras with infinite constant $\Psi_{\mathcal A}$}{Algebras with infinite constant Ψ(A)}}\label{sec:lp}

Below we  show that the numerical  {-}range spectral constant $\Psi(\cdot)$ (see \eqref{PsiT}) can attain the value $\infty$ as soon as we step away from $C^*$-algebras or matrix Banach algebras. We use the classical left- and right-shift operators and compute their algebraic numerical range $V(T)$ for completeness. 

\begin{theorem}\label{LR}
Let $p\in[1,\infty]$. The left-shift $L$ and right-shift $R$ satisfy:
\begin{enumerate}
    \item[\rm (i)]$V(L,\mathcal{B}(\ell^{p}))=V(R,\mathcal{B}(\ell^{p}))=\overline{\mathbb{D}}$;
    \item[\rm (ii)]$\Psi(L,\mathcal{B}(\ell^{p}))=\Psi(R,\mathcal{B}(\ell^{p}))=\begin{cases}
1&p=2\\
\infty&p\neq2
\end{cases}$.
\end{enumerate}
\end{theorem}
\begin{proof}
Suppose that $p\in[1,\infty)$ with H\"older conjugate $q\in(1,\infty]$, then $(\ell^{p})'=\ell^{q}$ (i.e.\ $\ell^{q}$ is the dual space of $\ell^{p}$) with $R'=L$ and using Lemma \ref{dualcro} we deduce that
\begin{align*}
V(R,\mathcal{B}(\ell^{p}))&=V(R',\mathcal{B}((\ell^{p})'))=V(L,\mathcal{B}(\ell^{q})),\\
\Psi(R,\mathcal{B}(\ell^{p}))&=\Psi(R',\mathcal{B}((\ell^{p})'))=\Psi(L,\mathcal{B}(\ell^{q})).
\end{align*}
For $p=\infty$ we have $\ell^{\infty}=(\ell^{1})'$ with $R=L'$ and therefore
\begin{align*}
V(R,\mathcal{B}(\ell^{\infty}))&=V(L',\mathcal{B}((\ell^{1})'))=V(L,\mathcal{B}(\ell^{1})),\\
\Psi(R,\mathcal{B}(\ell^{\infty}))&=\Psi(L',\mathcal{B}((\ell^{1})'))=\Psi(L,\mathcal{B}(\ell^{1})).
\end{align*}
Hence it suffices to prove both (i) and (ii) only for the left-shift $L$.

(i) Assume $p\in[1,\infty)$. Since $\nu(L,\mathcal{B}(\ell^{p}))\leq\|L\|_{p}=1$, we have $V(L,\mathcal{B}(\ell^{p}))\subseteq\overline{\mathbb{D}}$. To see that the reverse inclusion holds, define for each $\theta\in\mathbb{R}$ and $n\in\mathbb{N}$ the vector
\begin{align*}
x_{\theta,n}:=n^{-\tfrac{1}{p}}(1,e^{-i\theta},e^{-2i\theta},\ldots,e^{-(n-1)i\theta},0,\ldots)\in\ell^{p}
\end{align*}
and observe that
\begin{align*}
\lim_{\alpha\downarrow0}\frac{\|I_{\ell^{p}}+\alpha e^{i\theta} L\|_{p}-1}{\alpha}&\geq\lim_{\alpha\downarrow0}\frac{\|(I_{\ell^{p}}+\alpha e^{i\theta} L)x_{\theta,n}\|_{p}-1}{\alpha}\\
&=\lim_{\alpha\downarrow0}\frac{n^{-\frac{1}{p}}(1+(n-1)(1+\alpha)^{p})^{\frac{1}{p}}-1}{\alpha}=\frac{n-1}{n},
\end{align*}
which after letting $n\to\infty$ yields $\sup\operatorname{Re}e^{i\theta} V(L,\mathcal{B}(\ell^{p}))\geq1$ by the equality \eqref{r-theta}. The case $p=\infty$ follows from a similar but more direct argument using the vector $x_{\theta}:=(1,e^{-i\theta},0,\ldots)$ for $\theta\in\mathbb{R}$.

(ii) The case $p=2$ follows directly from von Neumann's inequality as $V(L,\mathcal{B}(\ell^{2}))=\overline{\mathbb{D}}$ and $\|L\|_{2}=1$. Assume $p\in[1,\infty]\setminus\{2\}$ and let $n\in\mathbb{N}$ be arbitrary. Let $P_{n}\colon\ell^{p}\to(\mathbb{C}^{n},\norm{\cdot}_{p})$ and $Q_{n}\colon(\mathbb{C}^{n},\norm{\cdot}_{p})\to\ell^{p}$ be defined by
\begin{align*}
    P_{n}(x_{1},\ldots,x_{n},x_{n+1},\ldots):=(x_{1},\ldots,x_{n}),\qquad Q_{n}(x_{1},\ldots,x_{n}):=(x_{1},\ldots,x_{n},0,\ldots).
\end{align*}
Clearly, $P_{n}Q_{n}=I_{\mathbb{C}^{n}}$ and $P_{n}LQ_{n}=J_{n}$ and $\|P_{n}\|=\|Q_{n}\|=1$. For every $f\in\mathbb{C}[z]$ we infer
\begin{align*}
P_{n}f(L)Q_{n}=f(P_{n}LQ_{n})=f(J_{n})
\end{align*}
and therefore $\|f(L)\|_{p}\geq\|f(J_{n})\|_{p}$. It follows that \begin{align*}
	\Psi(L,\mathcal{B}(\ell^{p}))\geq\sup_{\substack{f\in\mathbb{C}[z] \\ f\neq0}}\frac{\|f(J_{n})\|_{p}}{\|f\|_{\infty,\overline{\mathbb{D}}}}\geq\frac{1}{\sqrt{6}} n^{\max\{\frac{1}{p},\frac{1}{q}\}-\frac{1}{2}},
\end{align*}
by Theorem \ref{Jordan}. Thus $p\neq2$ implies $\Psi(L,\mathcal{B}(\ell^{p}))=\infty$ as desired. 
\end{proof}

Theorem~\ref{LR} says that both shifts  on $\ell^{p}$  ($p\neq2$) are examples of operators with operator norm $1$ and numerical  {-}range spectral constant $\infty$, the latter being due to the fact that they are both not polynomially bounded.  
In the next example we construct a polynomially bounded operator $T$ on a Banach space $X$ with $\|T\|=1$ and $\Psi(T)=\infty$.
\begin{example}\label{ex:polybddinfspectral}
    Consider the left-shift operator $L\colon\ell^{p}\to\ell^{p}$ for $p\neq2$. Endow $\mathbb{C}^{2}$ with the standard Hilbert norm and consider the algebraic direct sum
\begin{align*}
	X:=\mathbb{C}^{2}\oplus\ell^{p}.
\end{align*}
Equip $X$ with a Banach norm so that the induced operator norm on $\mathcal{B}(X)$ satisfies $\|A\oplus B\|=\max\{\|A\|_{2},\|B\|_{p}\}$ for all $A\in\mathcal{B}(\mathbb{C}^{2})$ and $B\in\mathcal{B}(\ell^{p})$. Consider the matrix
\begin{align*}
	E:=\begin{bmatrix}
		0&1\\
		0&0
	\end{bmatrix}
\end{align*}
and define the operator from $X$ to $X$ as
\begin{align*}
T:=E\oplus\tfrac{1}{2}L.
\end{align*}
It is clear that $\|T\|=1$. 
Since $V(E,\mathcal{B}(\mathbb{C}^{2}))=V(\tfrac{1}{2}L,\mathcal{B}(\ell^{p}))=\tfrac{1}{2}\overline{\mathbb{D}}$, it quickly follows that
\begin{align*}
\sup \operatorname{Re} e^{i\theta}V(T,\mathcal{B}(X))&=\lim_{\alpha\downarrow0}\frac{\|I_{X}-\alpha e^{i\theta} T\|-1}{\alpha}\\
&=\max\Big\{\lim_{\alpha\downarrow0}\frac{\|I_{\mathbb{C}^{2}}-\alpha e^{i\theta} E\|-1}{\alpha},\lim_{\alpha\downarrow0}\frac{\|I_{\ell^{p}}-\alpha e^{i\theta}\tfrac{1}{2}L\|-1}{\alpha}\}
=\frac{1}{2}
\end{align*}
for all $\theta\in\mathbb{R}$ and therefore $V(T,\mathcal{B}(X))=\tfrac{1}{2}\overline{\mathbb{D}}$ as well. 

Let us now prove that $T$ is polynomially bounded. By the von Neumann inequality we have $\|f(E)\|_{2}\leq\sup_{\overline{\mathbb{D}}}|f|$ for all polynomials $f$, while by Proposition \ref{prop:epsilonball} we have $\|f(\frac{1}{2}L)\|_{p}\leq \frac{2\sqrt 3}3\sup_{\overline{\mathbb{D}}}|f|$ for all polynomials $f$. We deduce that
\begin{align}\label{2sqrt3}
	\|f(T)\|=\max\{\|f(E)\|_{2},\|f(\tfrac{1}{2}L)\|_{p}\}\leq \frac{2\sqrt 3}3\sup_{\overline{\mathbb{D}}}|f|
\end{align}
for all polynomials $f$.

Finally, we show that $\Psi(T,\mathcal{B}(X))=\infty$. Since $V(\tfrac{1}{2}L,\mathcal{B}(\ell^{p}))=\tfrac{1}{2}\overline{\mathbb{D}}$ and $\Psi(\tfrac{1}{2}L,\mathcal{B}(\ell^{p}))=\infty$, there exists for each positive integer $n$ some polynomial $f_{n}$ such that
\begin{align*}
	\|f_{n}(\tfrac{1}{2}L)\|_{p}\geq n\sup_{\tfrac{1}{2}\overline{\mathbb{D}}}|f_{n}|.
\end{align*}
From this we deduce that
\begin{align*}
	\|f_{n}(T)\|\geq\|f_{n}(\tfrac{1}{2}L)\|_{p}\geq n\sup_{\tfrac{1}{2}\overline{\mathbb{D}}}|f_{n}|=n\sup_{V(T,\mathcal{B}(X))}|f_{n}|
\end{align*}
for all $n\in\mathbb{N}$ and the claim follows.  It remains an open question whether the constant $\frac{2\sqrt{3}}3$ in \eqref{2sqrt3} is optimal.   {In particular, it is unknown whether} there exists an operator $T$ of norm 1 which is polynomially bounded with constant 1 but with $\Psi(T)=\infty$.
\end{example}

We now turn our attention  to  combinatorial Banach spaces.
We will show that for a large class of these spaces the  universal spectral constant is infinite. Our idea is based on the \emph{spreading property}, hence it will include important examples such as the Schreier space and the Tsirelson space, see, e.g.  \cite{borodulin2024zoo} and the references therein.    
Let $\mathcal{F}$ be   {a} family of subsets of the positive integers $\mathbb{N}$, satisfying the following properties;
\begin{enumerate}
	\item every  $i \in \mathbb{N} $ belongs to some $F \in \mathcal{F}$;\label{cB0}
		\item if $\{l_1,l_2,\dots,l_n\} \in \mathcal{F}$ and $l_i \leqslant k_i$ for $k_{i}\in\mathbb{N}$ and all $i=1,\dots,n$, then  $\{k_1,k_2,\dots,k_n\}  \in \mathcal{F}$, (\textit{spreading property});\label{cB1}
        \item for each $ n\geq 1$ there exists $ F \in \mathcal{F}$ such that $|F|\geq n$.\label{cB2}
    \end{enumerate}
    See also \cite{Gowers}.
Consider the following norm on the space  $c_{00}$  of finitely supported sequences:
\begin{equation}
    \|x\|_{\mathcal{S}}:= \sup_{F \in \mathcal{F}} \sum_{i \in F} |x_i|,\quad x=(x_i)_{i=1}^\infty\in c_{00}.
\end{equation}
We define the combinatorial Banach space $\mathcal{S}$ as the completion of $c_{00}$ with respect to the above norm.
  Below we show  that the numerical  {-}range spectral constant \eqref{PsiA} of the algebra $\mathcal{B}(\mathcal S)$ is infinite.

\begin{theorem}\label{combinTh}
 Let $\mathcal S$ be a combinatorial Banach space, defined as above, satisfying \eqref{cB0}--\eqref{cB2}  {. Then}  $\Psi_{\mathcal{B}(\mathcal S)}=\infty$.
\end{theorem}

\begin{proof}
Let us fix a number $k \in \mathbb{N}$. By property \ref{cB2} of a combinatorial Banach space,  there exists  $n\geq k$ such that there exists a subset  $(k_1,\dots, k_n) \in \mathcal{F}$. Thanks to the spreading property we have   {that}
\begin{equation}\label{spread}
(k_n, k_n +1, \dots, k_n +n-1)\in\mathcal F.
\end{equation}
Let $P_n$ denote the projection onto the coordinates $k_n+1,\dots, k_n+n$ and $R$ be the right shift, both defined on $c_{00}$. Define the following  linear operator
$$
S_n = P_n \circ R,\quad S_n( (x_{j})_{j=1}^\infty)=(\underbrace{0,\dots,0}_{k_n}, x_{k_n},\dots,x_{k_n+n-1},0,0,\dots ).
$$
Observe that  $\norm{S_n}\leq 1$, hence, it extends to a bounded operator of norm not greater than 1 on the whole space $\mathcal S$. Indeed,  we have
$$
\norm{S_nx}_\mathcal{S} \leq \norm{S_nx}_{\ell_1}=\sum_{i=0}^{n-1}|x_{k_n+i}|\leq \norm{x}_{\mathcal S}, 
$$
where the last inequality follows due to \eqref{spread}.  

Further, observe that  $\norm{S_n+\lambda I}_\mathcal{S} \leq |\lambda| +1$ for $\lambda\in\Comp$.   {In fact, we have equality. To see this, take}  $x= e_{k_n}=(0, \dots, 0,1, 0,\dots)$ (unit on the $k_n$-th position) and note that it is a unit vector, due to \eqref{spread}. Hence,
$$
\norm{S_n+\lambda I}_{\mathcal S} =|\lambda| +1,\quad \lambda\in\Comp,
$$
again thanks to \eqref{spread}.
From this, together with \eqref{circleset}, we obtain that 
$ V(S_n)$
is the closed unit disk.

Let $f_n$ be the polynomials as in \eqref{Littlewood}.  From the form of $S_n$ we see that
$$
f_{n+1}(S_n)e_{k_n}  = (\underbrace{0,\dots, 0}_{k_n-1}, \alpha_0, \alpha_1,  \dots, \alpha_{n-1}, 0, 0, \dots),
$$
where $f_{n+1}(z)=\alpha_{0}+\alpha_{1}z+\dots+\alpha_{n}z^{n}$ and $\alpha_j\in\{-1,1\}$ ($j=0,\dots,n$).
Hence, using \eqref{spread} for the final time,
$$
\norm{f_{n+1}(S_n)}_\mathcal{S} \geq n,
$$
while $\sup_{z\in \overline{\mathbb{D}}}|f_{n+1}(z)| \leq \sqrt{6} \sqrt{n+1}$, which shows that $\Psi_{\mathcal{B}(\mathcal S)}=\infty$.

\end{proof}

\begin{remark}
There are several possibilities to extend the results of the current section using similar methods. First,  Theorem~\ref{combinTh} can be easily extended to higher order spaces, cf. \cite{antunes2021geometry,beanland2019extreme}.
   Second, one can show that for the algebra $\mathcal B(C(K))$  the spectral constant of the numerical range is infinite, under mild assumptions on $K$, in particular covering Theorem~\ref{LR} for $p=\infty$. We refrain from doing this, and in the subsequent section we concentrate on the analysis of the case where the numerical range is not necessarily a disk.
\end{remark}

\section{
\texorpdfstring{$\ell^1$-induced norm}{l1-induced norm}}\label{s:l1}

In this section we give explicit bounds of the numerical  {-}range spectral constant  $\Psi_{\mathcal A}$, see \eqref{PsiA}, for the algebra  of $2 \times 2$ matrices   {with the operator norm $\|\cdot\|_1$ induced by the $\ell^1$-norm on $\mathbb{C}^2$}. 
\begin{theorem}\label{13} The following inequalities hold:
    $1.1<\Psi_{\mathcal B(\Comp^{2}, \norm\cdot_1) }\leq 13$.
\end{theorem}

We divide the proof of this estimate into several steps.
{   First we show that the algebraic numerical range in the $\ell^1$-induced norm algebra is the Gershgorin column set.}  This natural result is crucial and, to the best of our knowledge, cannot be found in the literature.  The statement of Theorem~\ref{13} follows  directly from Lemma~\ref{jordan-1}, Lemma~\ref{jordan-2} (the upper bound) and Example~\ref{>1.1} (the lower bound).

\begin{theorem} \label{V1}
In the algebra  $\mathcal B(\ell^1,\norm\cdot_1)$  the algebraic numerical range of a bounded operator $T$ equals 
$$
V(T, \mathcal B(\ell^1,\norm\cdot_1) ) 
=\overline\conv \bigcup_{j=1}^{\infty} \bigg\{ D\bigg(t_{j,j},\ \sum_{k=1,  k\neq j}^{\infty} |t_{k,j}| \bigg) \bigg\},
$$
where $ (t_{k,j})_{k\in \mathbb N}:=Te_j$ and $e_1,e_2,\dots$ is the canonical  basis of $\ell^1$.  

Therefore, the algebraic numerical range of a matrix $T = [t_{i,j}]_{ij=1}^n$ in the algebra $\mathcal B(\Comp^{n\times n},\norm{\cdot}_{1})$  is given by the convex hull of the Gershgorin disks corresponding to its columns.
\end{theorem}

\begin{proof}
Before we proceed with the proof let us note that the operator norm of $T\in \mathcal B(\ell_1, \norm\cdot_1)$ can be calculated similarly to $\ell^1$-matrix norm. Namely, let $e_1,e_2,\dots$ be the canonical basis of $\ell^1$ and let $e_1^*,e_2^*,\dots$ denote their dual operators (i.e.\ the coefficient functionals corresponding to the Schauder basis $(e_{j})_{j=1}^{\infty}$). Define $t_{k,j} = e_k^* Te_j$  {. Then} the norm is given by

$$\norm{T}_1 = \sup_{j\in\mathbb N} \ \sum_{k\in\mathbb N} |t_{k,j}|.$$

Now let us fix an angle $\theta\in[0,2\pi)$. Our goal is to find the supporting hyperplane $H_{\theta}$ for the set $V(T)$ using formulas \eqref{hyperset}--\eqref{r-theta}.  
To simplify calculations let us rotate the coordinate complex plane by $\theta$, so that 
$$T' := [t_{i,j}']_{i,j\in\mathbb N} := e^{-i\theta}T, \qquad H' := H_0(T') = H_{\theta}(T), \qquad r':= r_0(T') = r_{\theta}(T).$$

Then the distance $r'$ can be expressed as 
\begin{eqnarray*}
 r' & = & \lim_{\alpha\to 0+}\frac{\norm{I + \alpha T'}_{1} - 1}{\alpha} \\
   & =& \lim_{\alpha\to 0+} \sup _{j\in \mathbb N} \left\{
    \frac{|1+ \alpha t_{j,j}'| - 1 + \sum_{k=1, k\neq j}^{\infty} | \alpha t_{k,j}'|}{\alpha} \right\}.
\end{eqnarray*}
For each $\alpha$ let us choose a sequence of indices $(j_m)_{m\in\mathbb N}$ approaching the supremum above and define a function

$$f(\alpha,m):= \frac{|1+ \alpha t_{j_m,j_m}'| - 1}{\alpha}  + \sum_{k=1, k\neq j_m}^{\infty} |t_{k,j_m}'|, \qquad a\in(0,1),\  m\in\mathbb N. $$
Note that the function $g(\alpha):=\frac{1}{\alpha}(\norm{I+\alpha T} - 1)$ is   {decreasing}, as for any $\alpha <\beta$ we have
$$
g(\alpha) = \frac{\norm{\frac{\beta}{\alpha} I + \beta T}_{1} - \frac{\beta}{\alpha} }{\beta} 
\leq \frac{\norm{I + \beta T}_{1} + \norm{\left( \frac{\beta}{\alpha} - 1\right)I} -\frac{\beta}{\alpha}  }{\beta} = g(\beta).
$$
Hence, we infer that
$$
r' = \lim_{n\to\infty}\lim_{m\to\infty} f(1/n,m).
$$
Next we show that $f(1/n,m)$ satisfies the conditions of the Moore-Osgood theorem to switch the order of the limits.
Note that $\lim\limits_{m\to\infty} f(1/n,m) = \frac{\norm{I + 1/n T'}_{1} - 1}{1/n} <\infty$  for each $n\in\mathbb N$ by the definition. To find the other limit observe that for an arbitrary complex number $z=a+bi$ we have

$$\lim_{n\to\infty}(|z + n| - n)
= \lim_{n\to\infty} \frac{a^2 + 2an + b^2}{\sqrt{(a+n)^2+b^2} + n} 
= a = \text{Re}(z).
$$
And consequently,

\begin{equation}\label{lim_n}
    \lim_{n\to\infty} f(1/n,m) = \sum_{k=1, k\neq j_m}^{\infty} |t_{k,j_m}'| + \text{Re}(t_{j_m,j_m}').
\end{equation}
Let us show that $\lim\limits_{n\to\infty} f(1/n,m)$ is also uniform in $m$. Observe that $|t'_{k,j}|\leq ||T||_1$ for all $k,j$.
For $\varepsilon >0$ let us choose $n>\frac{2}{\varepsilon}M^2$, where $M:= \max\{1,\norm{T}_1\}$. For simplicity let $t'_{j_m,j_m} = a_m+ib_m$ be the decomposition into real and imaginary parts. Then

\begin{eqnarray*}
0\leq f(1/n, m) - \lim_{n\to\infty} f(1/n,m) 
&=& \sqrt{(a_m+n)^2+b_m^2} - n - a_m \\
&= & \frac{b_m^2}{\sqrt{(a_m+n)^2+b_m^2} + n+a_m} 
\leq \frac{M^2}{n-M} \leq \varepsilon
\end{eqnarray*}
for all $m\in\mathbb N$. Hence, we can change the order of the limits in the definition of $r'$, which together with the equality \eqref{lim_n} provides
$$
r' = \lim_{m\to\infty} \left( \sum_{k=1, k\neq j_m}^{\infty} |t_{k,j_m}'| + \text{Re}(t_{j_m,j_m}') \right)
= \sup _{j\in \mathbb N} \left\{ \sum_{k=1, k\neq j}^{\infty} | t_{k,j}'| + \text{Re}( t_{j,j} ')\right\} . 
$$

Let us now fix  $j$ and consider the Gershgorin disk corresponding to the $j$-th column of $T'$, i.e., $D(t'_{j,j}, \sum_{k=1, k\neq j}^{\infty} | t_{k,j}'|)$. Let us look at its vertical tangent lines. If the Gershgorin disk is just a point there is only one such line passing through $t_{j,j}$, let us call it $l_{j,\theta}$.
Otherwise there are two such lines. Let us denote their touch points as $p_1$ and $p_2$ and without loss of generality assume that $\text{Re}(p_1) < \text{Re}(p_2)$. Then we denote the line corresponding to $p_2$ by $l_{j, \theta}$ (see Figure \ref{img_l_theta}). 

\begin{figure}[ht]
\centering
\includegraphics[width=.4\linewidth]{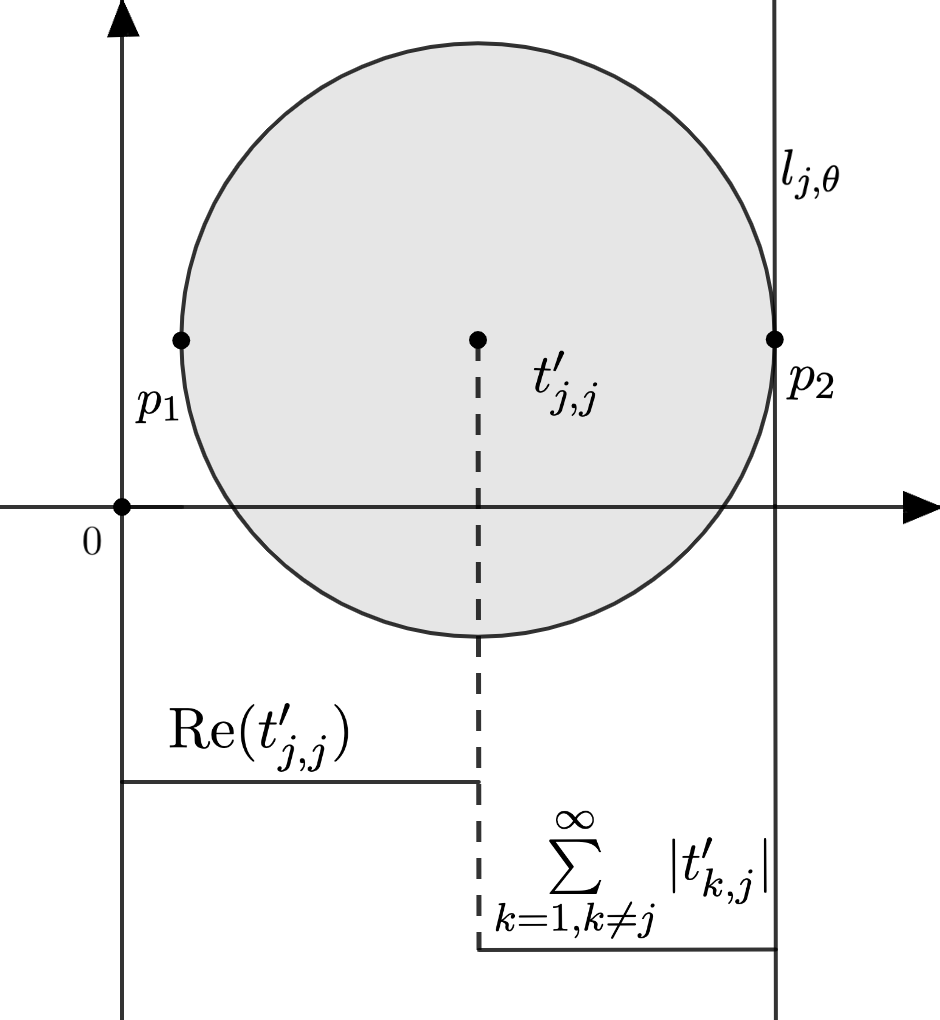} 
\hspace{0.3cm}
\includegraphics[width=.4\linewidth]{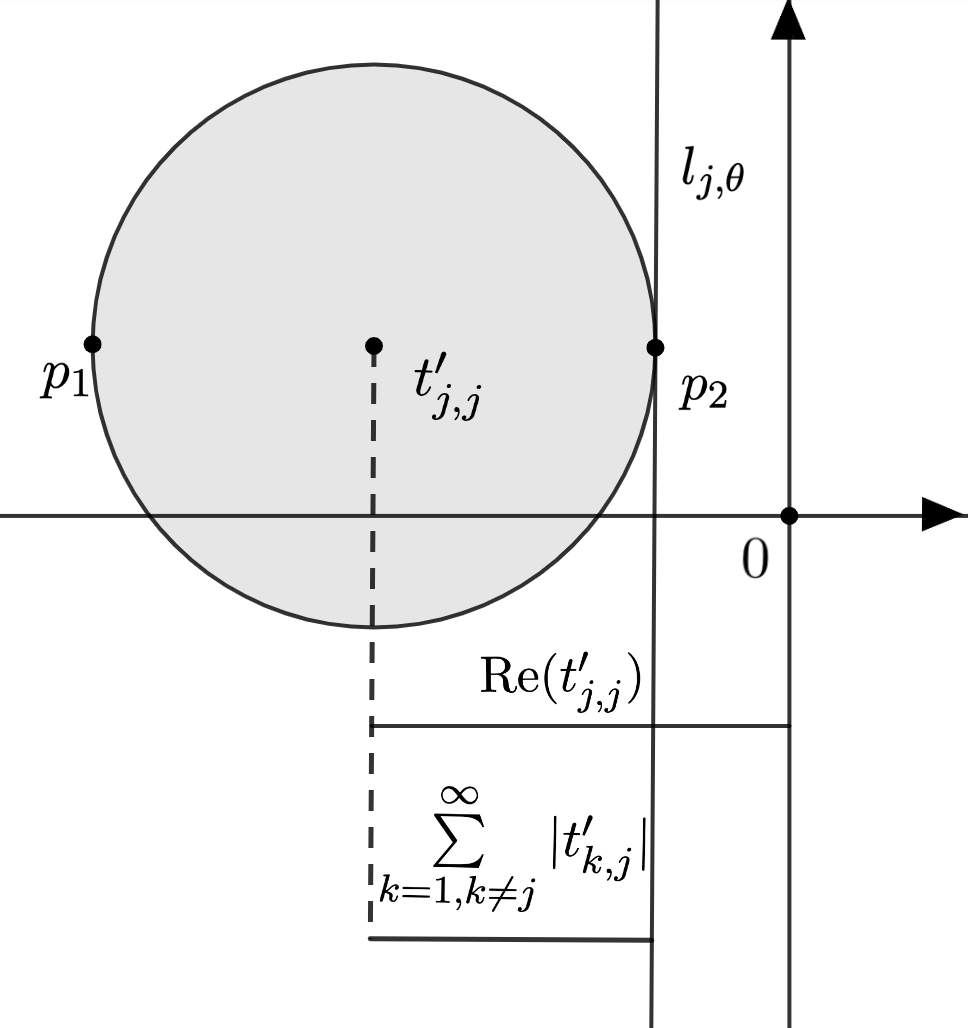} \par 
\caption{\small Illustration of the Gershgorin disk $D(t'_{j,j}, \sum_{k=1, k\neq j}^{\infty} | t_{k,j}'|)$ together the tangent line $l_{j,\theta}$. Its radius and the distance from 0 to its center is highlighted. The left picture shows an example when $\text{Re}(t_{j,j}')$ is positive and the right one corresponds to the negative case.}
\label{img_l_theta}
\end{figure}

Now it can be easily seen that the expression $\sum_{k=1, k\neq j}^{n} | t_{k,j}'| + \text{Re}(t_{j,j}')$ is equal to the distance from $0$ to the tangent line $l_{j,\theta}$, which, in turn, is equal to $\text{Re}(p_2)$.  
So, $r'$ is equal to the supremum of such distances from $0$ to $l_{j,\theta}$ over all $j\in \mathbb N$. Hence, half-plane $H'$ contains all Gershgorin disks and, moreover, it  is tangent to the closure of their convex hull. That means that $H' = H_{\theta}(T)$ is simultaneously a   {supporting} half-plane for the algebraic numerical range $V(T)$ and for the convex hull of the disks. Since both these sets are convex, they must be equal.
\end{proof}

\begin{remark}
{   Note that Theorem~\ref{V1} along with the inclusion $\sigma(T)\subseteq V(T)$ discussed in Section~\ref{s:prel},  provides  a broader explanation of why the Gershgorin set contains the eigenvalues. 
Additionally, it is worth  recalling the relationship between the Gershgorin sets and the classical numerical range, as discussed in \cite{Johnson1974,chang2013}}. Namely,  the classical numerical range is always contained in the convex hull of the union of the disks 
$$ 
D\bigg(t_{i,i},\ \frac{1}{2}\sum_{k=1,  k\neq i}^n |t_{i,k}|+|t_{k,i}|\bigg),
$$
which makes the latter set automatically a $(1+\sqrt 2)$-spectral set in the algebra $\mathcal B(\Comp^{n\times n},\norm\cdot_{\ell^2})$ and a $\sqrt{n}(1+\sqrt 2)$-spectral set in the algebra $\mathcal B(\Comp^{n\times n},\norm\cdot_{\ell^1})$.
It is, however, easy to verify that in general neither $V(T, \mathcal B(\mathbb{C}^{n},\norm\cdot_1)$  contains the classical numerical range nor conversely. 
\end{remark}

Now let us consider separate cases depending on the Jordan form of the $2\times 2$ matrix.

\begin{lemma}\label{jordan-1}
If $T\in \Comp^{2,2}$ is  similar to a Jordan block of size 2 then $\Psi(T; \mathcal B(\Comp^{2}, \norm\cdot_1 ) \leq 2+\sqrt 2$.
\end{lemma}

\begin{proof}

Observe that $T$  can be written in the following from
\begin{equation*}
     T =
     \begin{bmatrix}
     a & b \\
     c & d
     \end{bmatrix}
     \begin{bmatrix}
      x& 1 \\
      0& x
     \end{bmatrix}
     \begin{bmatrix}
      d& -b \\
      -c& a
     \end{bmatrix}
     , \quad a,b,c,d,x\in\mathbb{C}, \quad\text{with } ad-bc =1.
\end{equation*}
Since $\Psi(\alpha T + \beta I) = \Psi(T)$ (see Proposition \ref{semicont}), we can assume $x=0$. Then $T$ takes form
\begin{equation*}
     T =      
     \begin{bmatrix}
     -ac & a^2 \\
     -c^2 & ac
     \end{bmatrix}
\end{equation*}
and its algebraic numerical range $V(T)$ is given by the convex hull of two disks $D(-ac,|c|^2)$ and $D(ac,|a|^2)$.
Let us take a polynomial $p \in \mathbb{C} [z]$ such that $\sup\limits_{z\in V(T)} |p(z)|=1$. By a straightforward computation we have
\begin{equation*}
     p(T) =      
     \begin{bmatrix}
     a & b \\
     c & d
     \end{bmatrix}
     \begin{bmatrix}
      p(0) & p'(0) \\
      0& p(0)
     \end{bmatrix}
     \begin{bmatrix}
      d& -b \\
      -c& a
     \end{bmatrix}
     =
     \begin{bmatrix}
      p(0) - ac p'(0)& a^2 p'(0) \\
      -c^2 p'(0) & p(0) + ac p'(0)
     \end{bmatrix}.
\end{equation*}
and so 
    $ ||p(T)||_1 = \max \{ |p(0)  - ac p'(0)| + |c^2 p'(0)|,\ |p(0) + ac p'(0)| + |a^2 p'(0)| \} $.

Notice that   {the} disk $D\left(0,\frac{|a|^2 +|c|^2}{2}\right)$ is contained in the algebraic numerical range $V(T)$. This can be seen by calculating the midline of the trapezoid formed by a common tangent to two disks $D(-ac,|c|^2)$ and $D(ac,|a|^2)$ and radii drawn to this tangent (see Figure \ref{img_lemma_1_trapezoid}).
 So, $p$ is analytic in $D\left(0,\frac{|a|^2 +|c|^2}{2}\right)$ and bounded by 1. Hence, by Cauchy's inequality for the Taylor series coefficients of a complex analytic function, we get $|p'(0)| \leq \frac{2}{|a|^2+ |c|^2}$.

\begin{figure}[ht]
\centering
\includegraphics[width=.4\linewidth]{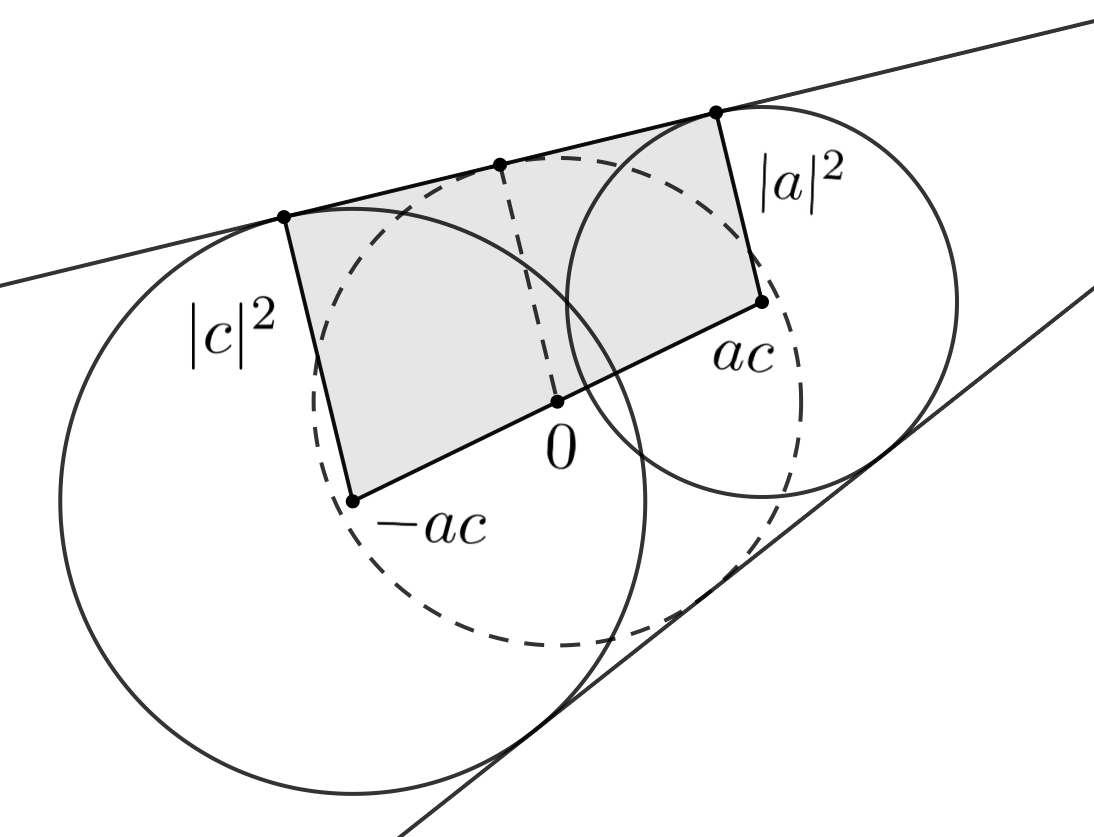} 
\caption{\small Illustration of the two disks $D(-ac,|c|^2)$ and $D(ac,|a|^2)$ with highlighted trapezoid formed by their common tangent and radii drawn to this tangent. The circle with dotted line illustrates the disk $D\left(0,\frac{|a|^2 +|c|^2}{2}\right)$ which is contained in the closure of the convex hull of $D(-ac,|c|^2)$ and $D(ac,|a|^2)$.} 
\label{img_lemma_1_trapezoid}
\end{figure}
 Therefore,     
\begin{equation*}
    \|p(T)\|_1 \leq |p(0)| + (|ac| + \max(|a|^2, |c|^2)) |p'(0)| 
    \leq 1 + 2\frac{|ac|+ \max(|a|^2, |c|^2)}{|a|^2+|c|^2}.
\end{equation*}
 Using the fact that the function $f(\alpha,\beta)=\frac{\alpha\beta+ \beta^2}{\alpha^2+\beta^2}$ is bounded by $(1+\sqrt{2})/2$ for $0<\alpha\leq \beta$ we obtain
$$
  \|p(T)\|_1 \leq2+\sqrt{2}
$$
as desired.
\end{proof}

\begin{lemma}\label{jordan-2}
    If $T\in \Comp^{2,2}$ is a diagonalizable matrix then
    $\Psi(T; \mathcal B(\Comp^{2}, \norm\cdot_1 ) \leq 13$.
    \end{lemma}
    \begin{proof}
  The matrix $T$ can be written in the form
\begin{equation*}
     T =      
     \begin{bmatrix}
     a & b \\
     c & d
     \end{bmatrix}
     \begin{bmatrix} x& 0 \\ 0& y
     \end{bmatrix}
     \begin{bmatrix}
      d& -b \\
      -c& a
     \end{bmatrix}
     , \quad a,b,c,d,x,y\in\mathbb{C}, \quad ad-bc =1.
\end{equation*}
As $\Psi(\lambda I) = 1$, let us assume $x\neq0$, $x\neq y$. Also, since $\Psi(T)=\Psi(\frac{1}{x}(T-yI))$, we can assume $x=1$ and $y=0$. Then
\begin{equation*}
     T = [a,c][d,-b]^T=     
       \begin{bmatrix}
     ad & -ab \\
     cd & -bc
     \end{bmatrix}.
\end{equation*}
Let us fix a polynomial $p \in \mathbb{C} [z]$ such that $\sup\limits_{z\in V(T)} |p(z)|=1$. Notice that 
\begin{equation*}
     p(T) =      
     \begin{bmatrix}
     a & b \\
     c & d
     \end{bmatrix}
     \begin{bmatrix}
      p(1) & 0\\
      0& p(0)
     \end{bmatrix}
     \begin{bmatrix}
      d& -b \\
      -c& a
     \end{bmatrix}
     =(p(1)-p(0))T+p(0)I,
\end{equation*}
and hence by the triangle inequality
\begin{equation}\label{norm}
    \norm{p(T)}_1 \leq |p(1)-p(0)| \norm{T}_1 +|p(0)|.
\end{equation}
Since $\sup_{V(T)}|p|\leq1$, we clearly have $|p(0)-p(1)|\leq2$ and $|p(0)|\leq1$. Thus, if $\|T\|_{1}\leq6$, then

\begin{equation*}
     \norm{p(T)}_1 \leq 2\cdot 6 +1 = 13.
\end{equation*}
Now assume that $\norm{T}_1> 6$.
Observe that the algebraic numerical range $V(T)$ is given by the convex hull of Gershgorin disks $D_1:=D(ad, |cd|)$ and $D_2:=D(-bc,|ab|)=D(1-ad,|ab|)$. Each of the eigenvalues $0$ and $1$ of $T$ belongs to at least one of the disks $D_{1}$ and $D_{2}$ by the Gershgorin circle theorem.

Define a polynomial $q\in \mathbb C[z]$ as $q(z):=p(z)-p(0)$. Notice that $q(0)=0$ and also $|q(z)| \leq |p(z)|+ |p(0)| \leq 2$ for $z\in V(T)$. Let $r_0$ denote the maximal radius such that $D(0,r_0) \subseteq V(T)$.   {If $r_0 \geq 1$, then, by Schwarz’s lemma,}
\begin{equation*}
    |p(1)- p(0)| = |q(1)| \leq \frac{2}{r_0}.
\end{equation*}
Otherwise, if $r_0<1$,   {then the} same inequality holds trivially, as then $|q(1)|\leq2 \leq \frac{2}{r_0}$.

Next we estimate $r_0$. 
Draw   {the} line $l$ tangent to disks $D_1$ and $D_2$ which is closer to 0 (see Figure \ref{img_lemma_2_m0}). Let $t_0$ and $t_{1/2}$ denote the   {orthogonal} projections of $0$ and $1/2$ onto $l$.
Consider the trapezoid formed by the centers of $D_{1}$ and $D_{2}$, and the points where $l$ meets the disks. 
  {Let us recall that $ad - bc = 1$, so the midpoint of $ad$ and $-bc$ is $(ad - bc)/2 = 1/2$.
Then it is easy to see that} the line segment with endpoints $1/2$ and $t_{1/2}$ is a midline of this trapezoid and hence $|t_{1/2} - 1/2 | = \frac{1}{2}(|ab|+|cd|)$.

Let $m_0$ denote the projection of $0$ onto the midline. Then we have $|t_0|=|t_{1/2}-m_0| = |t_{1/2} - 1/2| - |m_0-1/2|$. It is easy to see that $|m_0-1/2|\leq 1/2$ as it is a leg in a right triangle with hypotenuse of length $1/2$. Altogether we get 
\begin{equation}\label{jordan-r}
    r_0 = |t_0|\geq \frac{|ab|+|cd|-1}{2}
\end{equation}

\begin{figure}[ht]
\centering
\includegraphics[width=.6\linewidth]{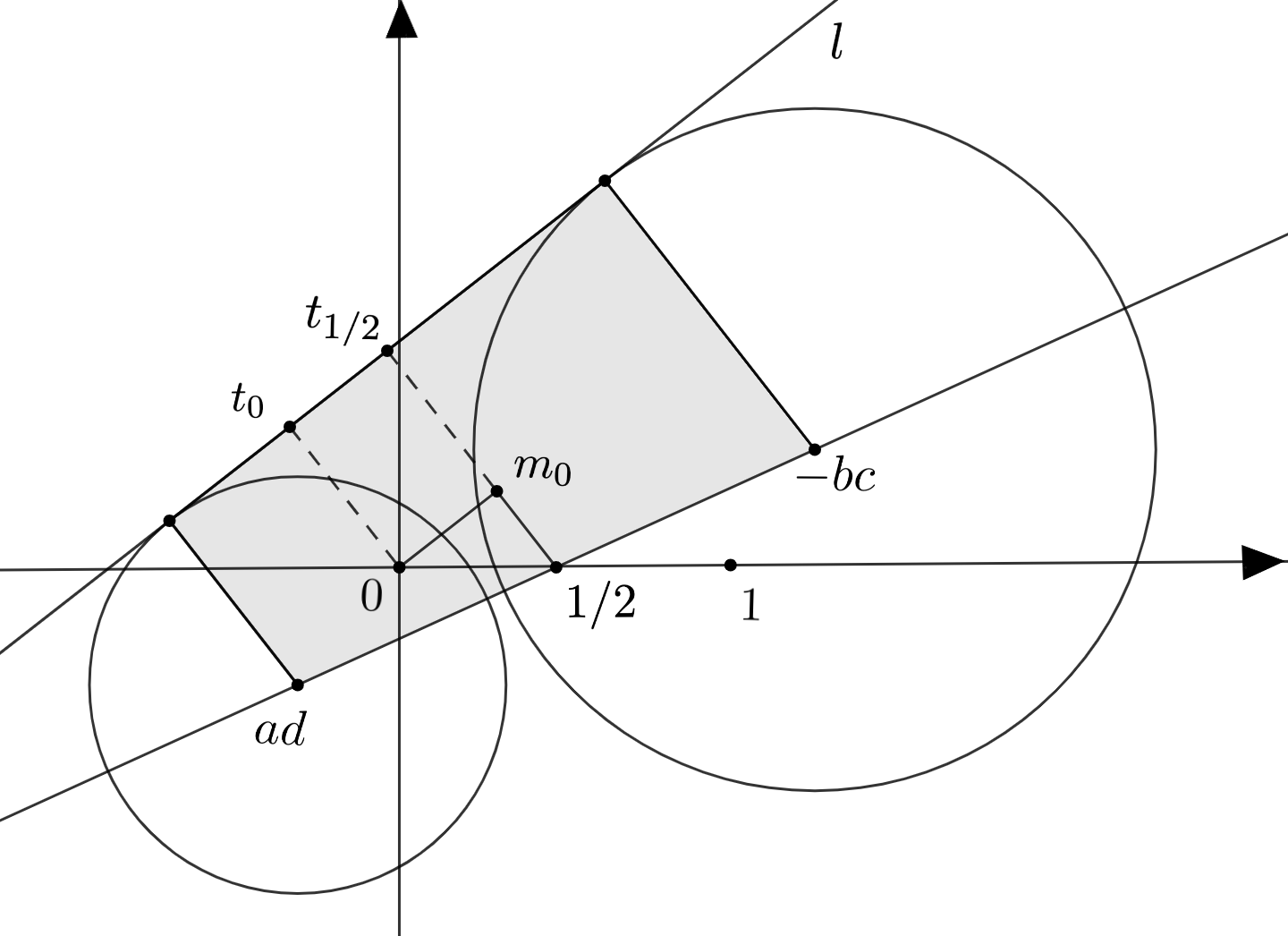} 
\caption{\small Illustration of the disks $D_1:=D(ad, |cd|)$ and $D_2:=D(-bc,|ab|)$ with highlighted trapezoid formed by their common tangent $l$ closest to 0 and radii drawn to it.}
\label{img_lemma_2_m0}
\end{figure}

Depending on the values of $|a|$, $|b|$, $|c|$ and $|d|$ we now distinguish between several cases.

\noindent\textit{Case 1:} $|a| \leq |c|$ and $|b| \leq |d|$, or equivalently $0,1\in D_1$.
\\In this case we have $\|T\|_{1}=|ad|+|cd|\leq2|cd|$. Inequality \eqref{norm} together with \eqref{jordan-r} gives
\begin{equation*}
     \norm{p(T)}_1 \leq 4 \frac{|ad|+|cd|}{|ab|+|cd|-1} + 1
     \leq 4 \frac{2|cd|}{|cd|-1} + 1.
\end{equation*}
Since we also assumed $\|T\|_{1}>6$, we have $|cd|>3$ and hence $\frac{2|cd|}{|cd|-1}<3$.
Altogether we get
\begin{equation*}
     \norm{p(T)}_1 \leq 4\cdot 3+1=13.
\end{equation*}

\noindent\textit{Case 2:} $|a| > |c|$ and $|b| > |d|$, or equivalently $0,1\in D_2$.  
\\This case is symmetrical to the previous \textit{Case 1}. Indeed, we have $6<\norm{T}=|ab|+|cb|\leq 2|ab|$ and
\begin{equation*}
     \norm{p(T)}_1 \leq 4 \frac{|ab|+|bc|}{|ab|+|cd|-1} + 1
     \leq 4 \frac{2|ab|}{|ab|-1} + 1
     \leq 13.
\end{equation*}

\noindent\textit{Case 3:}  $|a| \leq |c|$ and $|b| > |d|$, or equivalently $0\in D_1$ and $1\in D_2$. 
\\In this case inequality \eqref{norm} together with \eqref{jordan-r} gives
\begin{equation*}
     \norm{p(T)}_1 \leq 4 \frac{|ab|+|bc|}{|ab|+|cd|-1} + 1
     = 4 \frac{|ab|+|bc|}{(|ab|+|bc|)+(|cd|-|bc|)-1} + 1 .
\end{equation*}
Let us estimate difference $|bc|-|cd|$. 
Firstly $0<|bc|-|cd|$ due to our assumption $|b|>|d|$. On the other hand, since $|a|\leq |c|$ and $ad-bc=1$, by the triangle inequality we have $|bc|-|cd|\leq |bc|-|ad|\leq |bc-ad|=1$. So,
\begin{equation*}
     \norm{p(T)}_1 \leq 4 \frac{|ab|+|bc|}{|ab|+|bc|-2} + 1 \leq 7,
\end{equation*}
where in the last line we use the assumption $\norm{T}=|ab|+|bc|>6$ which implies $\frac{\norm{T}}{\norm{T}-2}\leq \frac{3}{2}$.

\noindent\textit{Case 4:}  $|a| > |c|$ and $|b| \leq |d|$, or equivalently $0\in D_2$ and $1\in D_1$. 
\\This case is symmetrical to \textit{Case 3}. Similar reasoning gives $0<|ad|-|ab|\leq |ad|-|bc|\leq |ad-bc|=1$ and
\begin{equation*}
     \norm{p(T)}_1 \leq 4 \frac{|ad|+|cd|}{(|ad|+|cd|)+(|ab|-|ad|)-1} + 1 \leq 4 \frac{|ad|+|cd|}{|ad|+|cd|-2} + 1 \leq 7.
\end{equation*}

\end{proof}

 Directly from Theorems~\ref{th:functions} and \ref{13} we receive the following explicit example of an   {infinite-dimensional} algebra, which is not a $C^*$-algebra, but nonetheless has a finite spectral constant.
\begin{corollary}
Let $X$ be a compact space, consider the Banach algebra $\mathcal A=C(X,\mathcal \Comp^{2,2})$ of matrix-valued functions on $X$ with the norm  
$\norm f:=\sup_{x\in X}\norm{ f(x)}_1$. Then
$
\Psi_{\mathcal A}\in[1.1, 13].
$  
\end{corollary}

The following example illustrates that the constant $\Psi_{\mathcal B(\mathbb C^2,\norm{\cdot}_1)}$ is greater than 1.

\begin{example}\label{>1.1}
Consider   {the} matrix $T=\begin{bmatrix}
2 & 1 \\
0 & 0
\end{bmatrix}$ and   {the} function $f(z) = \cos(z)$. 
  {Since }

$$T=\begin{bmatrix}
-1 & 1 \\
2 & 0
\end{bmatrix}
\begin{bmatrix}
0 & 0 \\
0 & 2
\end{bmatrix}
\begin{bmatrix}
0 & 1/2 \\
1 & 1/2
\end{bmatrix},
$$
  {a straightforward calculation shows that}

$$\norm{f(T)}_1=\max\left(|\cos(2)|,\frac{|\cos(2)-\cos(0)|}{2}+|\cos(0)|\right) > 1.708$$ 
  {and that} the algebraic numerical range is given by $V(T)=\text{conv}( \overline{\mathbb{D}} \cup \{2\} )$. On the other hand, notice that 
$$\cos(x+iy) =\cos(x) \cosh(y) - i\sin(x) \sinh(y)$$ 
and so 
$$|\cos(x+iy)|^2 = \cos(x)^2 (1+\sinh(y)^2) + \sin(x)^2\sinh(y)^2 = \cos(x)^2+\sinh(y)^2 .$$ 
Hence we have 
$$
\max_{z\in V(T)} |f(z)| \leq  \sqrt{1+\sinh(1)^2} < 1.55
$$
and 
$$
\Psi(T, \mathcal B(\mathbb C^2,\norm{\cdot}_1)> 1.708/1.55 > 1.1 .
$$
\end{example}

\section{Conclusions}

We have discussed the spectral constant of the numerical range for various Banach algebras.
Summarizing, we see three appearing questions for future research.

1. Is it true that $\Psi_{\mathcal A}<\infty$ for any matrix algebra $\mathcal A$?  {   More generally, is $\Psi(T)$  uniformly bounded for all Banach-algebra elements that are algebraic of a fixed degree $n$?}

2. Does there exist an operator $T$ which is polynomially bounded with constant 1, but with $\Psi(T)=\infty$?

3. Is it true that $\Psi(T)<\infty$  for all bounded operators on a combinatorial Banach space with the spreading property, in particular on the Schreier space?

\section*{Acknowledgment}

The Authors are indebted to Michael Hartz, Tomasz Kania and Anna Pelczar-Barwacz for fruitful discussions. {   The comments of the referees have led to a substantial improvement of the manuscript, which is here gratefully acknowledged.}

The research cooperation was funded by the program Excellence Initiative – Research University at the Jagiellonian University in Krak\'ow, which is here gratefully acknowledged.

The fourth named author has been financed by the Dutch Research Council (NWO) grant OCENW.M20.292.

%\begin{Backmatter}

\bibliographystyle{abbrvnat}
\bibliography{Literatura.bib}

%\printaddress

%\end{Backmatter}

\end{document}